\documentclass[reqno, 12pt]{amsart}
\pdfoutput=1
\makeatletter
\let\origsection=\section \def\section{\@ifstar{\origsection*}{\mysection}}
\def\mysection{\@startsection{section}{1}\z@{.7\linespacing\@plus\linespacing}{.5\linespacing}{\normalfont\scshape\centering\S}}
\makeatother

\usepackage{amsmath,amssymb,amsthm}
\usepackage{mathrsfs}
\usepackage{mathabx}\changenotsign
\usepackage{dsfont}

\usepackage{graphicx}
\usepackage{float}

\usepackage{xcolor}
\usepackage{hyperref}
\hypersetup{
	colorlinks,
	linkcolor={red!60!black},
	citecolor={green!60!black},
	urlcolor={blue!60!black}
}

\definecolor{codelightgray}{gray}{0.8}
\definecolor{codeverylightgray}{gray}{0.9}

\usepackage{array,multirow,colortbl,enumerate}
\usepackage{caption} 

\usepackage{graphicx}

\usepackage[open,openlevel=2,atend]{bookmark}

\usepackage[abbrev,msc-links]{amsrefs}
\usepackage{amsrefs}
\usepackage{doi}

\renewcommand{\PrintDOI}[1]{\doi{#1}}

\usepackage[T1]{fontenc}
\usepackage{lmodern}
\usepackage[babel]{microtype}
\usepackage[english]{babel}

\usepackage{geometry}
\numberwithin{equation}{section}
\numberwithin{figure}{section}

\usepackage{enumitem}
\def\rmlabel{\upshape({\itshape \roman*\,})}

\def\alabel{\upshape({\itshape \alph*\,})}

\let\polishlcross=\l
\def\l{\ifmmode\ell\else\polishlcross\fi}

\def\paragraph#1{	\noindent\textbf{#1.}\enspace}

\let\setminus=\smallsetminus

\let\sm=\setminus

\makeatletter
\def\moverlay{\mathpalette\mov@rlay}
\def\mov@rlay#1#2{\leavevmode\vtop{   \baselineskip\z@skip \lineskiplimit-\maxdimen
		\ialign{\hfil$\m@th#1##$\hfil\cr#2\crcr}}}
\newcommand{\charfusion}[3][\mathord]{
	#1{\ifx#1\mathop\vphantom{#2}\fi
		\mathpalette\mov@rlay{#2\cr#3}
	}
	\ifx#1\mathop\expandafter\displaylimits\fi}
\makeatother

\newcommand{\dcup}{\charfusion[\mathbin]{\cup}{\cdot}}
\newcommand{\bigdcup}{\charfusion[\mathop]{\bigcup}{\cdot}}

\DeclareFontFamily{U}  {MnSymbolC}{}
\DeclareSymbolFont{MnSyC}         {U}  {MnSymbolC}{m}{n}
\DeclareFontShape{U}{MnSymbolC}{m}{n}{
	<-6>  MnSymbolC5
	<6-7>  MnSymbolC6
	<7-8>  MnSymbolC7
	<8-9>  MnSymbolC8
	<9-10> MnSymbolC9
	<10-12> MnSymbolC10
	<12->   MnSymbolC12}{}
\DeclareMathSymbol{\powerset}{\mathord}{MnSyC}{180}

\usepackage{tikz}
\usetikzlibrary{calc,decorations.pathmorphing,decorations.pathreplacing}
\pgfdeclarelayer{background}
\pgfdeclarelayer{foreground}
\pgfdeclarelayer{front}
\pgfsetlayers{background,main,foreground,front}

\usepackage{subcaption}
\let\epsilon=\varepsilon

\let\rho=\varrho
\let\theta=\vartheta
\let\kappa=\varkappa

\def\NN{{\mathds N}}

\def\ZZ{{\mathds Z}}

\def\RR{{\mathds R}}

\def\dd{\mathrm{d}}

\theoremstyle{plain}
\newtheorem{thm}{Theorem}[section]
\newtheorem{theorem}[thm]{Theorem}

\newtheorem{prop}[thm]{Proposition}
\newtheorem{claim}[thm]{Claim}
\newtheorem{fact}[thm]{Fact}

\newtheorem{lemma}[thm]{Lemma}

\theoremstyle{definition}

\usepackage{accents}

\let\lra=\longrightarrow

\let\phi=\varphi

\begin{document}

\keywords{Measure theory, optimization, graph colourings}
\subjclass[2020]{28A25 (primary), 05C15, 05C35 (secondary)}
	
\title[Sets and partitions minimising small differences]{Sets and partitions minimising small differences}
	
\author[S. Antoniuk]{Sylwia Antoniuk}
\address{Department of Discrete Mathematics, Adam Mickiewicz University, Pozna\'n, Poland}
\email{antoniuk@amu.edu.pl}
	
\author[Chr. Reiher]{Christian Reiher}
\address{Fachbereich Mathematik, Universit\"at Hamburg, Hamburg, Germany}
\email{christian.reiher@uni-hamburg.de }
		
\begin{abstract}
	For a bounded measurable set $A\subseteq \RR$ we denote the Lebesgue measure 
	of $\{(x, y)\in A^2\colon x\le y\le x+1\}$ by $\Phi(A)$. We prove that if 
	$I=A_1\dcup\dots\dcup A_{k+1}$ partitions an interval $I$ of length $L$ into 
	$k+1$ measurable pieces, then 
	$\sum_{i=1}^{k+1} \Phi(A_i)\ge (\sqrt{k^2+1}-k)L-1$, where the multiplicative 
	constant $\sqrt{k^2+1}-k$ is optimal. 
	As a matter of fact we obtain the more general result 
	that $\Phi(A)\ge (\xi+\sqrt{1-2\xi+2\xi^2}-1)L-1$ whenever $A\subseteq I$ 
	has measure~$\xi L$. 
		\end{abstract}
	
	\maketitle

\section{Introduction}

\subsection{Colouring} 
The following discrete combinatorial question has recently surfaced 
in joint work of Dudek, Ruci\'nski, and the first author~\cite{ADR} 
on so-called ``randomly augmented graphs''. Given three integers $k\ge 0$
and $m, n\ge 1$, they ask for a colouring of the set $\{1, \dots, n\}$
with $k+1$ colours such that the number of pairs of integers $\{x, y\}$ 
which have the same colour and distance at most~$m$ is minimised; 
here one should imagine $n$ to be much larger than~$k$ and~$m$. If 
the desired minimum number of pairs could be determined with an accuracy 
of $\pm O_{k, m}(1)$, then some interesting results in extremal graph 
theory would follow. 

The approach to this problem pursued in this article is that it should be
easier to address the corresponding analytical problem first. This is suggested
by experience and analogy with other problems in extremal combinatorics 
of which we would like to mention the Lov\'asz-Simonovits triangle density 
problem~\cite{LoSi} as a prototypical example. In its original formulation the 
problem asks to determine, for given numbers $n$ and $e$, the least 
number~$t(n, e)$ of triangles that a graph on $n$ vertices with $e$ edges can 
contain. Despite a~recent breakthrough by Liu, Pikhurko, and Staden~\cite{LPS} this 
question is still open, whereas the corresponding analytical question to 
determine the function 
$f_3(\gamma)=\lim_{n\to\infty} t(n, \lceil \gamma n^2\rceil)/n^3$ has been
solved by Razborov~\cite{Raz} (see also~\cites{Niki, cdt} for generalisations 
to larger cliques). A perhaps more transparent, equivalent definition of
$f_3(\gamma)$, featured in Lov\'asz's book~\cite{Lov}, is that it is the 
minimum value of $\frac16\int_{[0, 1]^3}W(x, y)W(x, z)W(y, z)\,\dd x\,\dd y\,\dd z$
as $W$ varies over all symmetric measurable functions $[0, 1]^2\lra [0, 1]$
satisfying $\int_{[0, 1]^2}W(x, y)\,\dd x\,\dd y=2\gamma$. 

In order to formulate an analytical version of the colouring problem we mentioned 
in the beginning, we replace the discrete set $\{1, \dots, n\}$ by a real 
interval~$I$ of some given length $L$. Colourings of $I$ with $k+1$ colours are 
essentially the 
same as partitions $I=A_1\dcup\dots\dcup A_{k+1}$ of $I$ into $k+1$ classes. 
Instead of counting certain pairs of integers, we shall consider the measure of an 
analogous set, and thus we always need to require that the partition classes 
$A_1, \dots, A_{k+1}$ be Lebesgue measurable. One advantage of the current setup 
is that we can eliminate the parameter $m$ by rescaling it to $1$. Then the 
number of small-distance pairs with colour $i$ corresponds to the (two dimensional 
Lebesgue) measure $\Phi(A_i)$ of the set 
$\bigl\{(x, y)\in A_i^2\colon x\le y\le x+1\bigr\}$
and the question becomes how small the sum $\Phi(A_1)+\dots+\Phi(A_{k+1})$ can be. 
       
Let us look at a concrete example. For a given positive real number $t$ we first 
divide the interval $I=[0, L]$ into consecutive intervals 
$I_0=[0, t), I_1=[t, 2t), \dots$ of length $t$ and then for every 
$i=1, \dots, k+1$ we define $A_i$ to be the union of all intervals $I_j$ 
whose index~$j$ is congruent to~$i$ modulo~$k+1$. It can easily be shown that for 
$t\in \bigl[\frac 1{k+1}, \frac 1k\bigr]$ the sum $\Phi(A_1)+\dots+\Phi(A_{k+1})$
evaluates to $\frac12\bigl((k^2+1)t-2k+t^{-1}\bigr)L+O(1)$. The factor in front 
of~$L$ attains its least value $\sqrt{k^2+1}-k$ for $t=(k^2+1)^{-1/2}$. 
Theorem~\ref{thm:1} below asserts that this example is essentially optimal. 
Its statement avoids the error term $O(1)$ by addressing a circle $S=\RR/L\ZZ$ 
of perimeter $L\ge 1$ rather than an interval $I$ of length $L$. In this context, we can still define
\[
	\Phi(A)=\lambda\bigl(\{(x,y)\in A^2\colon x\leq y\leq x+1\}\bigr)
\]
for every measurable set $A\subseteq S$, but now the advantage is that pairs 
`wrapping around the circle' such as, for instance, the pairs in 
$[L-1/2, L)\times [0, 1/2)$ can contribute to $\Phi(A)$. 

\begin{theorem}\label{thm:1}    
	Let $k\ge 0$ be an integer and let $L\ge 1$ be real. 
	If $A_1\dcup\dots\dcup A_{k+1}$ partitions the circle $\RR/L\ZZ$
	into $k+1$ measurable parts, then
		\[
		\Phi(A_1)+\Phi(A_2)+\dots+\Phi(A_{k+1})
		\ge 
		(\sqrt{k^2+1}-k)L\,.
	\]
	\end{theorem}

As the discussion above suggests, this holds with equality if 
$n=\frac{\sqrt{k^2+1}L}{k+1}$ is an integer and every set $A_i$ consists of~$n$ 
intervals of length $\frac{1}{\sqrt{k^2+1}}$ that alternate cyclically 
(see Figure~\ref{Fig.2}).

\begin{figure}[h]
	
	\begin{center}
		\begin{tikzpicture}	
			\draw (0,0)	circle (1cm);	
			\draw [green,very thick,domain=70:110] plot ({cos(\x)}, {sin(\x)});	
			\draw [green,very thick,domain=190:230] plot ({cos(\x)}, {sin(\x)});	
			\draw [green,very thick,domain=310:350] plot ({cos(\x)}, {sin(\x)});	
			\draw [black,very thick,domain=110:150] plot ({cos(\x)}, {sin(\x)});	
			\draw [black,very thick,domain=230:270] plot ({cos(\x)}, {sin(\x)});	
			\draw [black,very thick,domain=350:390] plot ({cos(\x)}, {sin(\x)});	
			\draw [red,very thick,domain=150:190] plot ({cos(\x)}, {sin(\x)});	
			\draw [red,very thick,domain=270:310] plot ({cos(\x)}, {sin(\x)});	
			\draw [red,very thick,domain=30:70] plot ({cos(\x)}, {sin(\x)});
		\end{tikzpicture}
	\end{center}
	
	\caption{Optimal partition for $k+1=3$ and $n=3$} \label{Fig.2}
\end{figure}
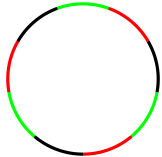

\subsection{Density} 
In Ramsey theory it frequently happens that a result on colourings of a~structure 
admits a generalisation to dense substructures. A well-known example where this 
phenomenon occurs is van der Waerden's theorem~\cite{vdW}, which states that if for 
a sufficiently large natural number $n=n(k, r)$ we colour the set $\{1, \dots, n\}$ 
with $r$ colours, then some $k$-term arithmetic progression will be monochromatic. 
The corresponding density result is a celebrated theorem of Szemer\'edi~\cite{Sz}
asserting that if $n=n(k, \delta)$ is sufficiently large, then every 
set $A\subseteq \{1, \dots, n\}$ of size $|A|\ge \delta n$ contains a $k$-term 
arithmetic progression. Clearly the case $\delta=1/r$ of Szemer\'edi's 
theorem implies van der Waerden's theorem, while, conversely, no `easy' way of 
deriving Szemer\'edi's theorem from van der Waerden's theorem is known. 
In fact, one obstacle for such a reverse deduction has recently been analysed 
in the PhD thesis of M.~Sales (see also~\cite{marcelo}).

This state of affairs indicates that one should at least wonder whether there 
is a~strengthening of Theorem~\ref{thm:1} that bounds the $k+1$ 
summands $\Phi(A_i)$ individually in terms of the densities $\xi_i=\lambda(A_i)/L$.
As it turns out, this can indeed be achieved. 

\begin{theorem}\label{thm:main}
	If $L\ge 1$, the set $A\subseteq \RR/L\ZZ$ is measurable, 
	and $\lambda(A)=\xi L$, then
		\[
		\Phi(A)\ge (\xi+\sqrt{2\xi^2-2\xi+1}-1)L\,.
	\]
	\end{theorem}

Note that equality holds if $n=\sqrt{2\xi^2-2\xi+1}L$ is an integer and~$A$ consists 
of $n$ equally spread intervals, each of length 
$\frac{\xi L}n$ (see Figure~\ref{Fig.1}). In particular, for every $\xi\in[0,1]$ the factor $\xi + \sqrt{2\xi^2-2\xi+1}-1$ in Theorem~\ref{thm:main} is optimal. 

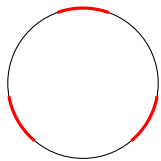
\begin{figure}[h]
\begin{center}
	\begin{tikzpicture}	
	\draw (0,0)circle (1cm);	
	\draw [red,very thick,domain=70:110] plot ({cos(\x)}, {sin(\x)});	
	\draw [red,very thick,domain=190:230] plot ({cos(\x)}, {sin(\x)});	
	\draw [red,very thick,domain=310:350] plot ({cos(\x)}, {sin(\x)});	
	\end{tikzpicture}
\end{center}
\caption{Three equidistant intervals of length $\frac{\xi L}{3}$.} 
\label{Fig.1}
\end{figure}

We conclude this introduction by showing that the density version of our 
result implies the colouring version. 

\begin{proof}[Proof of Theorem~\ref{thm:1} assuming Theorem~\ref{thm:main}]
	Let $\xi_1,\dots, \xi_{k+1}$ be the numbers defined in such a way that $\lambda(A_i)=\xi_i L$ holds for every $i$. Then $\xi_i$'s sum up to $1$ and, therefore, Theorem~\ref{thm:main} yields
		\[
		\sum_{i=1}^{k+1}\Phi(A_i) 
		\geq 
		\sum_{i=1}^{k+1} (\xi_i+\sqrt{\xi_i^2+(1-\xi_i)^2}-1)L
		=
		\biggl(\sum_{i=1}^{k+1}\sqrt{\xi_i^2+(1-\xi_i)^2}\biggr)L - kL\,.
	\]
		Minkowski's inequality implies 
		\[
		\sum_{i=1}^{k+1}\sqrt{\xi_i^2+(1-\xi_i)^2}
		\ge
		\sqrt{\biggl(\sum_{i=1}^{k+1}\xi_i\biggr)^2 + 
			\biggl(\sum_{i=1}^{k+1}(1-\xi_i)\biggr)^2} 
		=  
		\sqrt{k^2+1}\,,
	\]
	and thus we have indeed $\sum_{i=1}^{k+1}\Phi(A_i)\ge (\sqrt{k^2+1}-k)L$.
\end{proof}

\subsection*{Organisation} We prove Theorem~\ref{thm:main} in the next 
two sections and conclude by offering some remarks on the discrete problem 
in~Section~\ref{sec:conc}.

\section{Preliminaries}

Let $W\colon [0, 1]\lra \RR$ be the function defined by 
\[
	W(\xi) 
	=
	\sqrt{2\xi^2-2\xi+1} 
	= 
	\sqrt{\xi^2+(1-\xi)^2}
\]
for every $\xi\in [0, 1]$. When $\xi$ is clear from the 
context, we will often write $W$ instead of~$W(\xi)$. 
Obviously, this function is symmetric in the sense that 
$W(1-\xi)=W(\xi)$ holds for all $\xi\in [0, 1]$. In the 
left half of the unit interval we will sometimes use the following 
estimate. 

\begin{fact}\label{f:21}
	If $\xi\in [0, \frac12]$, then $2W\le 2-\xi$.
\end{fact}

\begin{proof}
	Due to $7\xi < 4$ this follows from
	$4W^2 = 8\xi^2-8\xi+4=(2-\xi)^2 + \xi(7\xi-4)$.
\end{proof}

For the rest of this and the next section we fix a real number $L\ge 1$
and we write   
\[
	S = \RR/L\ZZ
\]
for the circle with circumference $L$. For every measurable set $A\subseteq S$ 
we define the function $f_A\colon S\lra [0, 1]$ 
by $f_A(x)=\lambda([x, x+1]\cap A)$. In view of Fubini's theorem, the 
quantity~$\Phi(A)$ defined in the introduction rewrites as 
\[
	\Phi(A)=\int_A f_A(x)\,\dd x\,.
\]
If $A$ has measure $\lambda(A)=\xi L$, then Theorem~\ref{thm:main} states 
that the difference 
\[
	\eta(A)=\Phi(A)-(\xi+W-1)L
\]
is nonnegative. Our first lemma reduces this claim to the case $\xi\le \frac12$.

\begin{lemma}\label{lemma:1}
	If $S=A\dcup B$ is a partition of $S$ into two measurable 
	sets, then $\eta(A)=\eta(B)$.
\end{lemma}

\begin{proof}
Suppose $\lambda(A)=\xi L$ for some $\xi\in[0,1]$. By Fubini's theorem we have
\begin{equation}\label{eq:1}
	\int_S f_A(x) \, \dd x = \xi L\,.
\end{equation}
Indeed, both sides of the above equation are equal to the measure of the set
\[
	\{(x,y)\in S\times A\colon x\leq y \leq x+1\}\,.
\]
Next, since $f_A(x)+f_B(x)=1$ holds for all $x\in S$, we also know that
\[
	\int_B f_A(x)\, \dd x + \int_B f_B(x) \, \dd x 
	= 
	(1-\xi)L\,.
\]
By subtracting this from equation~\eqref{eq:1}, we infer that
\begin{align*}
	\Phi(A)-\Phi(B) 
	& = 
	\int_A f_A(x) \, \dd x - \int_B f_B(x) \, \dd x 
	= 
	\int_S f_A(x) \, \dd x - \int_B f_A(x) \, \dd x - \int_B f_B(x) \, \dd x \\
	& = 
	(2\xi-1)L 
	= 
	(\xi + W -1)L - (W-\xi)L\,,
\end{align*}
whence $\eta(A)=\eta(B)$.
\end{proof}

For a given measurable set $A\subseteq S$ we can render the definition of $\Phi(A)$ more symmetrical by considering the function $g_A\colon S\longrightarrow [0,2]$
defined by $g_A(x)=f_A(x)+f_A(x-1)$ for every~$x\in S$. 
In the special case $L\ge 2$ this rewrites as 
\[
	g_A(x) 
	= 
	\lambda([x-1,x+1]\cap A)\,,
\]
whereas for $L\in [1, 2]$ one needs to read this equation with the understanding 
that the interval $[x-1, x+1]$ wraps around the circle $S$ so that some parts of 
$A$ might be ``taken into account twice''. More precisely, if $\pi\colon\RR\lra S$ denotes 
the canonical projection $x\longmapsto x+L\ZZ$, then  
\[
	g_A(x) 
	= 
	\lambda\bigl([x'-1,x'+1]\cap \pi^{-1}[A]\bigr)\,,
\]
where on the right hand side $x'\in \RR$ denotes an arbitrary representative 
of $x$, and $\lambda$ refers to the Lebesgue measure on $\RR$. We will 
sometimes benefit from the obvious fact that for every interval $I\subseteq S$
and every $x\in I$ we have 
\begin{equation}\label{eq:gix}
	g_I(x)\ge \min\{\lambda(I), 1\}\,.
\end{equation}

In general, $\Phi(A)$ can be obtained from $g_A$ in the following way.

\begin{lemma}
	For every measurable set $A\subseteq S$ we have
		\[
		\int_A g_A(x) \,\dd x 
		= 
		2\Phi(A)\,.
	\]	
	\end{lemma}

\begin{proof}
	It suffices to show 
		\[ 
		\Phi(A)
		= 
		\int_A f_A(x-1) \, \dd x\,.
	\]
		Here the right side is equal to the measure of the set 
	\[ 
		\bigl\{(x,y)\in A^2\colon x-1\leq y\leq x\bigr\} 
		= 
		\bigl\{(x,y)\in A^2\colon y\leq x\leq y+1\bigr\}\,,
	\]
		which, by symmetry, coincides with $\Phi(A)$.
\end{proof}	

Preparing a compactness argument we will now treat the special 
case of Theorem~\ref{thm:main}, where the set $A$ is a union of 
finitely many intervals and, depending on the number of these 
intervals, the perimeter $L$ of the circle $S$ is sufficiently 
large. The precise bound $12n$ obtained here will be irrelevant 
in the future --- it only matters that some such number, depending 
only on $n$, exists.
   
\begin{lemma}\label{lemma:2}
	Let $n$ be a positive integer. If $A\subseteq S$ is a union 
	of $n$ mutually disjoint intervals and $L\ge 12n$, 
	then $\eta(A)\ge 0$.
\end{lemma}

\begin{proof}
	Since $S\setminus A$ is a union of $n$ disjoint intervals 
	as well, Lemma~\ref{lemma:1} allows us to suppose 
	that $\lambda(A)=\xi L$ holds for some $\xi\in [0,\frac12]$. 
	Let $A=I_1\dcup I_2\dcup\dots\dcup I_n$, where for every 
	$i\in [n]$ the set $I_i$ is an interval of length $\alpha_i$.
	Distinguishing between `long' and `short' intervals, we consider 
	the partition $[n]=P\dcup Q$ defined by 
		\[
		P=\bigl\{i\in[n]\colon \alpha_i \ge 2\xi\bigr\} 
		\quad \text{ and } \quad 
		Q=\bigl\{i\in[n]\colon \alpha_i < 2\xi\bigr\}\,.
	\]
		The total length of the short intervals can be bounded by
	$\sum_{i\in Q}\alpha_i\le 2\xi|Q|\le 2\xi n\le \frac16\xi L$,
	and thus we have
		\[ 
		\sum_{i\in P}\alpha_i\ge \frac56 \xi L\,.
	\]
		Combined with the fact that for all $i\in P$ and $x\in I_i$ 
	equation~\eqref{eq:gix} yields	
		\[
		g_A(x)
		\ge 
		g_{I_i}(x) 
		\ge 
		\min\bigl\{1,\alpha_i\bigr\} 
		\ge 
		2\xi\,,
	\] 
		this reveals 
		\begin{equation}\label{eq:1147}
		\Phi(A) 
		\ge
		\frac12 \sum_{i\in P}\int_{I_i}g_A(x) \,\dd x
		\ge 
		\xi \sum_{i\in P}\alpha_i
		\ge 
		\frac56\xi^2 L\,.
	\end{equation}
		
	Since $\xi\in[0,\frac12]$ implies
		\begin{align*}	
		W+1-\xi 
		&= 
		\sqrt{2\xi^2-2\xi+1} + (1-\xi) 
		= 
		\sqrt{2\left(\xi-\frac12\right)^2 + \frac12} + (1-\xi) \\
		& \ge 
		\sqrt{\frac12} + \frac12 
		> 
		\frac7{10} + \frac12 = \frac65\,,
	\end{align*}
		we have 
		\[
		\xi^2
		=
		W^2-(1-\xi)^2
		=
		(W+1-\xi)(W+\xi-1)
		\ge 
		\frac65(W+\xi-1)\,.
	\]
		Therefore,~\eqref{eq:1147} leads to the desired estimate
	$\Phi(A)\geq (W+\xi-1)L$.
\end{proof}

Recall that a real function $h$ defined on an interval is said to 
be $1$-\emph{Lipschitz} 
if
\[
	|h(x)-h(y)| \le |x-y|
\]
holds for all $x$ and $y$ in the domain of $h$. Clearly, if a function $h$ 
has this property, then $x\longmapsto h(x)+x$ is nondecreasing and, therefore,
measurable. This shows that all $1$-Lipschitz functions are measurable. 

\begin{lemma}\label{lemma:3}
	If $\xi\in(0,\frac12]$, $\alpha\geq 0$, and $h\colon [0,\alpha]\lra \RR$ 
	is a $1$-Lipschitz function satisfying
	\begin{enumerate}[label=\rmlabel]
		\item\label{it:24i} $h(0)=h(\alpha)=1+\frac{2\xi-1}{W}$,
		\item\label{it:24ii} and 
			$h(x)\geq \max\bigl\{\frac{\xi}{W},\alpha\bigr\}$ 
			for all $x\in[0,\alpha]$, 
	\end{enumerate}
	then
	\[
		\int_0^{\alpha} h(x)\, \dd x
		\geq 
		\frac{2\alpha(W+\xi-1)}{\xi}\,.
	\]
\end{lemma}

\begin{proof}
Due to
\[
	(W+1)(W+\xi-1) 
	= 
	W^2 + W\xi + (\xi-1) 
	= 
	\xi(2\xi-1+W)\,,
\]
the quantity  
\[
	\phi 
	= 
	1 + \frac{2\xi-1}{W}
\]
appearing in~\ref{it:24i} factorises as
\begin{equation}\label{eq:3}
	\phi = \frac{(W+1)(W+\xi-1)}{W\xi}\,.
\end{equation}
Together with $W=\sqrt{\xi^2+(1-\xi)^2}>1-\xi$, this shows 
$\phi>0$. Depending on $\alpha$ we shall distinguish three 
cases. In all of them, we exploit that our assumptions on $h$ 
lead to a pointwise lower bound $\bar{h}\le h$, which is piecewise 
linear (see Figure~\ref{fig:25}). 
Thus, the integration $\int_0^\alpha \bar{h}(x)\,\dd x$
can easily be carried out and it will only remain to establish a 
linear or quadratic inequality in $\alpha$. 

\begin{figure}[ht]
	\centering
	
	\begin{subfigure}[b]{0.32\textwidth}
		\centering
	\begin{tikzpicture}	\draw[->] (-0.5,0) -- (3.5,0) coordinate (x axis);
	\draw[->] (0,-0.3) -- (0,3.5) coordinate (y axis);
	\draw (1.6,-0.05) -- (1.6,0.05) (3,-0.05) -- (3,0.05) (-0.05,3) -- (0.05,3);
	
	\draw[dashed] (1.6,0) -- (1.6,3);
	\draw[thick, domain=0:0.8, smooth, variable=\z, blue] plot ({\z}, {3-\z});
	\draw[thick, domain=0.8:1.6, smooth, variable=\z, blue] plot ({\z}, {1.4+\z});
	
	\node[below] at (1.6,0) {$\alpha$};
	\node[below] at (3,0) {$\phi$};
	\node[left] at (0,3) {$\phi$};
	\node[below] at (3.5,0) {$x$};
	\node[left] at (0,3.5) {$\bar{h}(x)$};
			
	\end{tikzpicture}
		\caption{$\alpha\le 2\bigl(1-\frac{1-\xi}{W}\bigr)$}
		\label{fig:25a}
	
\end{subfigure}
\hfill    
\begin{subfigure}[b]{0.32\textwidth}
	\centering
		\begin{tikzpicture}	
			
			\draw[->] (-0.5,0) -- (3.5,0) coordinate (x axis);
			\draw[->] (0,-0.3) -- (0,3.5) coordinate (y axis);
			\draw (2.2,-0.05) -- (2.2,0.05) (3,-0.05) -- (3,0.05) (-0.05,2.5) -- (0.05,2.5) (-0.05,3) -- (0.05,3);
			
			\draw[dashed] (2.2,0) -- (2.2,3);
			\draw[thick, domain=0:0.5, smooth, variable=\z, blue] plot ({\z}, {3-\z});
			\draw[thick, domain=0.5:1.7, smooth, variable=\z, blue] plot ({\z}, {2.5});
			\draw[thick, domain=1.7:2.2, smooth, variable=\z, blue] plot ({\z}, {0.8+\z});
			
			\node[below] at (2.2,0) {$\alpha$};
			\node[below] at (3,0) {$\phi$};
			\node[left] at (0,2.5) {$\frac{\xi}{W}$};
			\node[left] at (0,3) {$\phi$};
			\node[below] at (3.5,0) {$x$};
			\node[left] at (0,3.5) {$\bar{h}(x)$};
			
		\end{tikzpicture}
		\caption{$ 2\bigl(1-\frac{1-\xi}{W}\bigr)\leq\alpha\leq \frac{\xi}{W}$}
		\label{fig:25b}
	\end{subfigure}
\hfill    
\begin{subfigure}[b]{0.32\textwidth}
	\centering
		\begin{tikzpicture}
			\draw[->] (-0.5,0) -- (3.5,0) coordinate (x axis);
			\draw[->] (0,-0.3) -- (0,3.5) coordinate (y axis);
			\draw (2.7,-0.05) -- (2.7,0.05) (3,-0.05) -- (3,0.05) (-0.05,2.7) -- (0.05,2.7) (-0.05,3) -- (0.05,3);
			
			\draw[dashed] (2.7,0) -- (2.7,3);
			\draw[thick, domain=0:0.3, smooth, variable=\z, blue] plot ({\z}, {3-\z});
			\draw[thick, domain=0.3:2.4, smooth, variable=\z, blue] plot ({\z}, {2.7});
			\draw[thick, domain=2.4:2.7, smooth, variable=\z, blue] plot ({\z}, {0.3+\z});
			
			\node[below] at (2.7,0) {$\alpha$};
			\node[below] at (3,0) {$\phi$};
			\node[left] at (0,2.7) {$\alpha$};
			\node[left] at (0,3) {$\phi$};
			\node[below] at (3.5,0) {$x$};
			\node[left] at (0,3.5) {$\bar{h}(x)$};
			
		\end{tikzpicture}
	
	\caption{$\frac{\xi}{W} \le\alpha$} 
	\label{fig:25c}
\end{subfigure}
\caption{The function $\bar{h}$.}
\label{fig:25}
\end{figure}
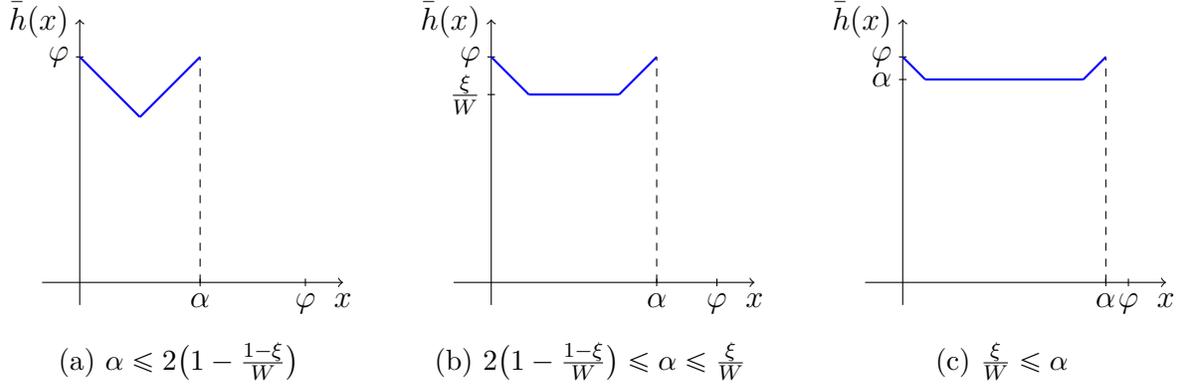

\smallskip
{\hskip2em \it First case: 
$\alpha \leq 2\bigl(1-\frac{1-\xi}{W}\bigr)$}

\smallskip

Since $h$ is $1$-Lipschitz,~\ref{it:24i} yields that $h\ge\bar{h}$, 
where
\[
\bar{h}(x)=\begin{cases}
	\phi-x & \text{ if }x\in [0,\alpha/2]\\
	\phi-(\alpha-x) & \text{ if } x\in [\alpha/2,\alpha]
\end{cases}
\]
(see Figure~\ref{fig:25a}).
Elementary geometric area considerations show 
\[ 
	\int_0^{\alpha}h(x)\,\dd x 
	\geq 
	\int_0^{\alpha}\bar{h}(x)\,\dd x 
	= 
	\alpha\phi - \frac{\alpha^2}4\,.
\]
As our upper bound on $\alpha$ and Fact~\ref{f:21} imply
\begin{align*}
	\phi - \frac{\alpha}4 
	& \overset{\eqref{eq:3}}{\ge} 
	\frac{(W+1)(W+\xi-1)}{W\xi} - \frac{W+\xi-1}{2W} 
	= 
	\frac{(2W+2-\xi)(W+\xi-1)}{2W\xi} \\
	& \overset{\phantom{\eqref{eq:3}}}{\ge}
	\frac{4W(W+\xi-1)}{2W\xi}
	= 
	\frac{2(W+\xi-1)}{\xi}\,,
\end{align*}
we thus have indeed
\[
	\int_0^{\alpha} h(x)\,\dd x 
	\ge 
	\frac{2\alpha(W+\xi-1)}{\xi}\,.
\]

\smallskip

{\hskip2em \it Second case: 
$2\left(1-\frac{1-\xi}{W}\right) \le \alpha \leq \frac{\xi}{W}$}

\smallskip

This time we use additionally that $h(x)\ge \frac\xi W$ holds for 
all $x\in [0, \alpha]$, which leads to the refined estimate 
$h(x)\geq \bar{h}(x)$, where
\[
\bar{h}(x)=
\begin{cases}
	\phi - x & \text{ if } x\in\bigl[0,1-\frac{1-\xi}{W}\bigr] \\
	\frac{\xi}{W} & \text{ if } x\in \bigl[1-\frac{1-\xi}{W},
		\alpha-\left(1-\frac{1-\xi}{W}\right)\bigr]	\\
	\phi - (\alpha-x) & \text{ if } x\in 
		\bigl[\alpha-\left(1-\frac{1-\xi}{W}\right),\alpha\bigr]
\end{cases}
\]
is visualised in Figure~\ref{fig:25b}. 
Decomposing the area below $\bar{h}$ into a rectangle and two 
right-angled isosceles triangles we obtain
\[
	\int_0^{\alpha} h(x)\,\dd x 
	\ge 
	\frac{\xi}{W}\alpha + \left(\phi-\frac{\xi}{W}\right)^2\,.
\]

Since the definitions of $\phi$ and $W$ entail
\begin{align*}	\left(\phi-\frac{\xi}{W}\right)^2 
	&=
	\left(1-\frac{1-\xi}{W}\right)^2 
	= 
	1 - \frac{2(1-\xi)}{W} + \frac{W^2-\xi^2}{W^2}   \\
	&= 
	\frac{2(W+\xi-1)}{W} - \frac{\xi^2}{W^2}
	= 
	\frac{\xi}{W}\left(\frac{2(W+\xi-1)}{\xi} - \frac{\xi}{W}\right)\,, \end{align*}
we are led again to 
\[
	\int_0^{\alpha} h(x)\, \dd x 
	\ge 
	\frac{\xi}{W}\alpha+
		\left(\frac{2(W+\xi-1)}{\xi} - \frac{\xi}{W}\right)\alpha
	=
	\frac{2\alpha(W+\xi-1)}{\xi}\,. 
\]

\smallskip

{\hskip2em \it Third case: $\alpha \geq \frac{\xi}{W}$}

\smallskip

We contend that 
\begin{equation}\label{eq:1842}
	\phi \ge \alpha\ge \frac23\phi\,.
\end{equation}
The first inequality follows from the fact that our assumptions
on $h$ imply $\phi=h(0)\ge \alpha$. Using Fact~\ref{f:21} we also 
obtain $3\alpha W\ge 3\xi=(2-\xi)+2(2\xi-1)\ge 2(W+2\xi-1)=2W\phi$,
which confirms the second estimate.  

Similar to the previous case we have $h(x)\geq \bar{h}(x)$
for all $x\in [0, \alpha]$, where
\[
\bar{h}(x)=
\begin{cases}
	\phi-x & \text{ if } x\in [0,\phi-\alpha]\\
	\alpha & \text{ if } x\in [\phi-\alpha,2\alpha-\phi]\\
	\phi-(\alpha-x) & \text{ if } x\in[2\alpha-\phi,\alpha]
\end{cases}
\]
is the function drawn in Figure~\ref{fig:25c}. By~\eqref{eq:1842}
all three of the above closed intervals are nonempty, whence
\begin{align*}
	\int_0^{\alpha} h(x)\,\dd x 
	&\ge 
	\int_0^{\alpha} \bar{h}(x)\,\dd x 
	=
	\alpha^2+(\phi-\alpha)^2 \\
	&=
	2\left(\alpha-\frac\xi W\right)^2
		+\left(\frac{4\xi}W-2\phi\right)\alpha
		+\left(\phi^2-\frac{2\xi^2}{W^2}\right)\,. 
\end{align*}

Owing to 
\[
	\frac{4\xi}W-2\phi
	=
	\frac{4\xi-2(W+2\xi-1)}{W}
	=
	\frac{2(1-W)}W
\]
and 
\begin{align*}
	\phi^2-\frac{2\xi^2}{W^2} 
	= 
	\frac{(1-2\xi)^2-2(1-2\xi)W+W^2-2\xi^2}{W^2} 	= 
	\frac{2(1-2\xi)(1-\xi-W)}{\xi W}\cdot \frac\xi W \le 0\,,
\end{align*}
we can conclude that
\[
	\int_0^{\alpha} h(x)\,\dd x 
	\ge 
	2\left(\frac{1-W}W+\frac{(1-2\xi)(1-\xi-W)}{W\xi}\right)
	\alpha\,.
\]
Finally, because of 
\begin{align*}
	\frac{1-W}W&+\frac{(1-2\xi)(1-\xi-W)}{W\xi}
	=
	\frac 1W-1+\frac{W^2-\xi-(1-2\xi)W}{W\xi} \\
	&=
	-1+\frac{W+2\xi-1}{\xi}
	=
	\frac{W+\xi-1}{\xi}\,,
\end{align*}
this simplifies to the desired inequality. 
\end{proof}		
	
\section{The proof}

Lemma~\ref{lemma:3} immediately yields the following special 
case of Theorem~\ref{thm:main}.

\begin{lemma}\label{lemma:4}	
Let $A\subseteq S$ be a union of intervals of measure 
$\lambda(A)=\xi L$, where $\xi\in(0,\frac12]$. If
\begin{enumerate}[label=\alabel]
	\item\label{it:31a} $g_A(x)=1+\frac{2\xi-1}{W}$ whenever $x$ 
		is a boundary point of~$A$,
	\item\label{it:31b} and $g_A(x)\ge \frac{\xi}{W}$ 
		for all $x\in A$, 
\end{enumerate}
then $\eta(A)\geq 0$.
\end{lemma}	

\begin{proof}
	Write $A=I_1 \dcup I_2 \dcup \dots \dcup I_n$, where for every 
	$i\in[n]$ the set $I_i$ is an interval of length~$\alpha_i$. 
	We shall prove
		\begin{equation}\label{eq:5}
		\int_{I_i}g_A(x)\,\dd x 
		\ge 
		\frac{2(W+\xi-1)}{\xi}\alpha_i	
	\end{equation}		
		for every $i\in [n]$. The addition of these estimates will yield
		\[
		2\Phi(A) 
		= 
		\int_A g_A(x)\,\dd x 
		\ge 
		\frac{2(W+\xi-1)}{\xi} \sum_{i=1}^n \alpha_i 
		= 
		2(W+\xi-1)L\,, 
	\]
		which in turn implies $\eta(A)\ge 0$.

	For the verification of~\eqref{eq:5} we can assume, 
	by rotational symmetry, that~$I_i=[0,\alpha_i]$. 
	If~$\alpha_i \geq 1$, then~\eqref{eq:gix} yields $g_A(x) \ge 1$ for 
	every $x\in I_i$, so that Fact~\ref{f:21} implies
		\[ 
		\int_{I_i}g_A(x) \,\dd x 
		\ge 
		\alpha_i 
		\ge 
		\frac{\xi+2W-(2-\xi)}{\xi}\alpha_i 
		= 
		\frac{2(W+\xi-1)}{\xi}\alpha_i\,.
	\]
		
	From now on we can suppose $\alpha_i\le 1$, which leads to
	$g_A(x)\ge \alpha_i$ for every $x\in I_i$. Our plan is to show
	that the restriction $g_A\restriction {I_i}$ satisfies 
	the assumptions of Lemma~\ref{lemma:3}. As~\ref{it:24i} 
	and~\ref{it:24ii} follow from~\ref{it:31a} and~\ref{it:31b},
	it remains to prove that $g_A\restriction I_i$ is $1$-Lipschitz.
	Whenever $0\le x\le y\le \alpha_i$ we have indeed 
		\begin{align*}
		\big|g_A(x)-g_A(y)\big|
						&= 
		\big|\lambda([x-1,y-1]\cap A)-\lambda([x+1,y+1]\cap A)\big| \\
		&\le 
		\max\bigl\{\lambda([x-1,y-1]),\lambda([x+1,y+1])\bigr\} 
		= 
		y-x\,. \qedhere
	\end{align*}
	\end{proof}

Our next result removes the apparently quite restrictive 
assumptions~\ref{it:31a} and~\ref{it:31b} from the previous lemma.
 
\begin{prop}\label{prop:32}
	If $A\subseteq S$ is a union of intervals, then $\eta(A)\ge 0$.
\end{prop}	

In the proof below it will be important to consider the perimeter $L$ 
of $S$ as one of the variables. Thus, the pair $(A, S)$ can be described 
by an $(2n+1)$-tuple of real numbers and we can study a potential ``worst 
counterexample'' by means of calculus. Roughly speaking, it turns out that 
such a configuration, if it existed, would need to satisfy the hypotheses of 
Lemma~\ref{lemma:4}, which is absurd.  

\begin{proof}[Proof of Proposition~\ref{prop:32}]
	Suppose that the proposition is not true and let $n$ denote the least integer    	such that there is a counterexample $(A, S)$, where $A$ is a union of $n$ 		
	intervals. 

	Consider the set
		\[ 
		\Delta 
		= 
		\bigl\{(a_1,a_2,\ldots,a_{2n},L)\in\RR^{2n+1}\colon
		0\le a_1\le a_2\le\dots\le a_{2n} \le L 
		\text{ and } 1\le L\le 12n+1\bigr\}\,.
	\]
		With every point $p=(a_1,a_2,\dots,a_{2n},L)\in\Delta$ we 
	associate the set $A^p\subseteq S=\RR/L\ZZ$ defined by
		\[
		A^p=[a_1,a_2)\dcup [a_3,a_4)\dcup\dots\dcup[a_{2n-1},a_{2n})\,.
	\]
		Our choice of $n$ and Lemma~\ref{lemma:2} ensure that there exists 
	some point $p\in \Delta$ such that $\eta(A^p)<0$. 
	Since $\Delta$ is compact and $p\longmapsto\eta(A^p)$ is a continuous 
	function from $\Delta$ to $\RR$, the extreme value theorem of Weierstra\ss\ 
	yields a point $p_\star\in\Delta$ such that $\eta_\star=\eta(A^{p_\star})$ 
	is minimal. By our indirect assumption $\eta_\star$ is negative.

	Let us write $p_\star=(a_1,a_2,\dots,a_{2n},L)$ and $A=A^{p_\star}$. 
	The complement $S\sm A$ corresponds to the point 
	$(a_2, \dots, a_{2n}, a_1, L)$ of $\Delta$, and in view of Lemma~\ref{lemma:1} 
	it satisfies $\eta(S\sm A)=\eta_\star$. 
	So if necessary we can still replace $A$ 
	by $S\sm A$, and for this reason we may suppose that the density $\xi$ 
	defined by $\lambda(A)=\xi L$ is in $\bigl[0, \frac12\bigr]$. 
	Due to $\eta(\varnothing)=0>\eta(A)$ we actually 
	have $\xi\in \bigl(0, \frac12\bigr]$.
	
	The minimality of $n$ implies that $A$ cannot be expressed as a union 
	of fewer than $n$ intervals, whence
		\begin{enumerate}
		\item[$\bullet$] $a_1<a_2<\dots<a_{2n}$,
	\end{enumerate}
		and if $a_1=0$, then $a_{2n}\neq L$. An appropriate rotation of the circle $S$
	allows us to strengthen the latter statement to
		\begin{enumerate}
		\item[$\bullet$] $0<a_1$ and $a_{2n}<L$.
	\end{enumerate}
		\begin{claim}\label{claim:1}
		For every $i\in[2n]$ we have $g_A(a_i)=1+\frac{2\xi-1}{W}$\,.
	\end{claim}

	\begin{proof}
		For reasons of symmetry it suffices to display the 
		argument for $i=1$. 
		Given a real number $\delta$ with
				\[
			|\delta|<\min\{a_1,a_2-a_1\}
		\]
				we set
				\[
			p(\delta)=(a_1-\delta,a_2,\dots,a_{2n},L) 
			\quad \text{ and } \quad A^{\delta}=A^{p(\delta)}\,.
		\]
				Notice that $p(0)=p_\star$ and $A^0=A$. 
		The minimality of $\eta_\star$ yields 
		$\eta_\star\le\eta(A^{\delta})$, whence
				\[
			\Phi(A^{\delta})-\Phi(A) 
			\geq 
			(\xi'+W'-1)L-(\xi+W-1)L
			=
			(\xi'-\xi+W'-W)L
		\]	
				holds for
				\[
			\xi'
			=
			\frac{\lambda(A^{\delta})}{L}
			=
			\frac{\xi L+\delta}{L} 
			= 
			\xi + \frac{\delta}{L}
			\quad \text{ and } \quad 
			W'=W(\xi')\,.
		\]
		 		Taylor's formula shows
				\[
			W' 
			= 
			\sqrt{2(\xi')^2-2\xi'+1} 
			= 
			W+\frac{2\xi-1}{W}(\xi'-\xi)+O\left((\xi'-\xi)^2\right),
		\]
				and thus we obtain
				\[
			\Phi(A^{\delta})-\Phi(A) 
			\geq
			(\xi'-\xi+W'-W)L
			=
			\delta\left(1+\frac{2\xi-1}{W}\right)+O(\delta^2)\,.
		\]
				Together with the equation
				\begin{equation}\label{eq:1959}
			\Phi(A^{\delta})-\Phi(A) 
			=
			\delta g_A(a_1)+O(\delta^2)\,,
		\end{equation}
				which we shall prove below, this will yield 
		$\delta\bigl(g_A(a_1)-1-\frac{2\xi-1}W\bigr)\ge O(\delta^2)$,
		and our claim follows by taking $\delta\to 0^-$ and $\delta\to 0^+$. 

		Now it remains to confirm~\eqref{eq:1959}. We will 
		only deal with the case $\delta>0$ here, the case $\delta<0$ 
		being similar. Setting 
				\[
			D=\{(x, y)\in S^2\colon x\le y\le x+1\}
			\quad \text{ and } \quad 
			K=[a_1-\delta, a_1)\,,
		\]
				we have $A^\delta=K\dcup A$, wherefore 
				\begin{align*}
			\Phi(A^{\delta})-\Phi(A)
			&=
			\lambda\bigl((K\dcup A)^2\cap D\bigr)-\lambda(A^2\cap D) \\
			&=
			\lambda(A\times K\cap D)+\lambda(K\times A\cap D)
			+\lambda(K^2\cap D)\,.
		\end{align*}
				Due to 
				\begin{align*}
			\lambda(A\times K\cap D)
			&=
			\delta\lambda(A\cap [a_1-1, a_1])+O(\delta^2)\,, \\
			\lambda(K\times A\cap D)
			&=
			\delta\lambda(A\cap [a_1, a_1+1])+O(\delta^2)\,, \\
			\text{ and } \quad
			\lambda(K^2\cap D)&=O(\delta^2)\,,
		\end{align*}
				this establishes~\eqref{eq:1959}.
	\end{proof}	

	\begin{claim}\label{claim:2}
		For all $x\in A$ we have $g_A(x)\geq\frac{\xi}{W}$.
	\end{claim}

\begin{proof}
	It suffices to show $g_A(x)\geq\frac{\xi}{W}$ whenever 
	$x\in [a_{2n-1}, a_{2n})$. For every $\delta\in[0,1]$ we set
		\[ 
		p(\delta)=(a_1,a_2,\dots,a_{2n-1},a_{2n}+\delta,L+\delta) 
		\quad \text{ and } \quad 
		A^{\delta}=A^{p(\delta)}\,.
	\]
		Due to $L\le 12n$ we have $p(\delta)\in\Delta$ and, therefore, 
	$\eta(A^{\delta})\geq \eta(A)$.
	This means that the numbers 
		\[
		\xi'=\frac{\lambda(A^{\delta})}{L+\delta} 
		\quad \text{ and } \quad 
		W'=W(\xi') 
	\]
		satisfy
		\[ 
		\Phi(A^{\delta})-\Phi(A)
		\geq 
		(\xi'+W'-1)(L+\delta) - (\xi+W-1)L\,. 
	\]
		In order to simplify the right side we observe that
		\[
		\xi' 
		= 
		\frac{\xi L + \delta}{L+\delta} 
		= 
		\xi+\frac{1-\xi}{L}\delta + O(\delta^2)
	\]
		and
		\[ 
		W' 
		= 
		W + \frac{2\xi-1}{W}(\xi'-\xi)+O((\xi'-\xi)^2) 
		= 
		W + \frac{(2\xi-1)(1-\xi)}{LW}\delta+O(\delta^2)
	\]
		lead to
		\begin{align*}
		& \phantom{==} (\xi'+W'-1)(L+\delta) - (\xi+W-1)L \\
		& = 
		(\xi'-\xi+W'-W)(L+\delta) + (\xi+W-1)\delta	\\
		& = 
		\left(\left(1+\frac{2\xi-1}{W}\right)(1-\xi)+	
			(\xi+W-1)\right)\delta+O(\delta^2) \\
		& = 
		\left(\frac{(2\xi-1)(1-\xi)}{W}+\frac{W^2}{W}\right)\delta+O(\delta^2) \\
		& = 
		\frac{\xi}{W}\delta + O(\delta^2)\,.
	\end{align*}
	For this reason it suffices to show 
		\begin{equation}\label{eq:1419}
		\Phi(A^{\delta})-\Phi(A)
		\le
		\delta g_A(x)+O(\delta^2)\,.
	\end{equation}
		The direct comparison of $A^\delta$ and $A$ becomes more transparent if we 
	rotate both sets by $L-x$ in their respective circles. That is, we look instead 
	at the sets 
		\begin{align*}
		B
		&=
		[0, a_{2n}-x)\dcup \bigdcup_{1\le i<n} [a_{2i-1}+L-x, a_{2i}+L-x)
			\dcup [a_{2n-1}+L-x, L) \\
		\text{ and } \quad 
		B^\delta
		&=
		[0, a_{2n}-x)\dcup \bigdcup_{1\le i<n} [a_{2i-1}+L-x, a_{2i}+L-x)
			\dcup [a_{2n-1}+L-x, L+\delta)\,,
	\end{align*}
		which satisfy $\Phi(B)=\Phi(A)$ as well as $\Phi(B^\delta)=\Phi(A^\delta)$.
	Lifting the whole situation to $\RR^2$ we thus obtain 
		\[
		\Phi(A)=\lambda\bigl(B^2\cap (D_1\dcup D_2)\bigr)\,,
	\]
		where 
		\begin{align}
		D_1&=\{(x, y)\in [0, L)^2\colon x\le y\le x+1\} \label{eq:3D1} \\
		\text{ and } \quad 
		D_2&=\{(x, y)\in [L-1, L)\times [0, 1]\colon y\le x+1-L\} \label{eq:3D2} \,,
	\end{align}
	and, similarly, 
		\[
		\Phi(A^\delta)=\lambda\bigl((B^\delta)^2\cap (D^\delta_1\dcup 
			D^\delta_2)\bigr)\,,
	\]
		where 
		\begin{align*}
		D^\delta_1&=\{(x, y)\in [0, L+\delta)^2\colon x\le y\le x+1\} \\
		\text{ and } \quad 
		D^\delta_2
		&=
		\{(x, y)\in [L+\delta-1, L+\delta)\times [0, 1]\colon y\le x+1-L-\delta\}\,.
	\end{align*}
		The fact that $B^\delta$ consists of $B$ and the interval $[L, L+\delta)$
	implies 
		\begin{align}\label{eq:1446}
		\lambda\bigl((B^\delta)^2\cap D^\delta_1\bigr)-\lambda\bigl(B^2\cap D_1) 
		&=
		\lambda\bigl((B^\delta)^2\cap [L-1, L]\times[L, L+\delta)\bigr)+O(\delta^2)
		\notag \\
		&=
		\delta\lambda(B\cap [L-1, L))+O(\delta^2) \notag \\
		&=
		\delta f_A(x-1)+O(\delta^2)\,.
	\end{align}
		Furthermore, $D^\delta_2\sm D_2\subseteq [L, L+\delta)\times [0, 1]$
	yields 
		\begin{align*}
		\lambda\bigl((B^\delta)^2\cap D^\delta_2\bigr)-\lambda\bigl(B^2\cap D_2) 
		&\le
		\lambda\bigl((B^\delta)^2\cap [L, L+\delta)\times[0, 1]\bigr)\\
		&=
		\delta\lambda(B\cap [0, 1])\\
		&=
		\delta f_A(x)\,.
	\end{align*}
		By adding this to~\eqref{eq:1446} we infer 
		\[
		\Phi(A^\delta)-\Phi(A)
		\le 
		\delta\bigl(f_A(x-1)+f_A(x)\bigr)+O(\delta^2)\,,
	\]
		which proves~\eqref{eq:1419}.
\end{proof}

	The Claims~\ref{claim:1}~and~\ref{claim:2} show that $A$ satisfies the 
	assumptions of Lemma~\ref{lemma:4}. So $\eta(A)\geq 0$, which contradicts 
	the choice of $A$.
\end{proof}

\begin{proof}[Proof of Theorem~\ref{thm:main}]
	Owing to the regularity of the Lebesgue measure, there exists a~sequence 
	of sets $(A_n)_{n\in \NN}$, each a union of finitely many intervals, 
	that approximates $A$ in the sense that 
	$\lim_{n\to\infty}\lambda(A_n\triangle A)=0$. As this entails 
	$\lim_{n\to\infty}\lambda((A_n)^2\triangle A^2)=0$, 
	we have $\lim_{n\to\infty} \Phi(A_n)=\Phi(A)$. Together with 
	$\lim_{n\to\infty} \lambda(A_n)=\lambda(A)$ and the continuity of~$W$, 
	this leads to $\lim_{n\to\infty} \eta(A_n)=\eta(A)$. By Proposition~\ref{prop:32} 
	each term of this sequence is nonnegative and thus $\eta(A)$ cannot be 
	negative either.
\end{proof}

\section{A concluding remark}\label{sec:conc}

Let us finally return to the discrete problem investigated by Dudek, Ruci\'{n}ski,
and the first author~\cite{ADR}. Given a finite set $A\subseteq \ZZ$ and a positive 
integer~$m$ we write $\Psi_m(A)$ for the number of pairs $(x, y)\in A^2$ such 
that $x<y\le x+m$. For any three integers $k\ge 0$ and $m, n\ge 1$, the 
minimum value of $\Psi_m(A_1)+\dots+\Psi_m(A_{k+1})$ taken over all partitions 
$[n]=A_1\dcup\dots\dcup A_{k+1}$ is denoted by $f(k, m, n)$. 

In graph theoretic term this function can be viewed as follows. The $m^\mathrm{th}$
power of the path on $n$ vertices is a graph $P^m_n$ with vertex set 
$[n]$ and all edges $\{x, y\}$ satisfying $|x-y|\le m$. 
With every vertex colouring $\phi\colon V(P_n^m)\lra [k+1]$ we associate the 
partition $[n]=A_1\dcup\dots\dcup A_{k+1}$ defined by $A_i=\phi^{-1}(i)$ for 
every $i\in [k+1]$. The sum $\Psi_m(A_1)+\dots+\Psi_m(A_{k+1})$ is then the same 
as the number of edges of $P^m_n$ which connect vertices of the same colour. 
Thus, $f(k, m, n)$ is the minimum number of monochromatic edges of $P^m_n$ that
a $(k+1)$-colouring of $V(P^m_n)$ can have.

It is plain that $f(k, m, n_1+n_2)\le f(k, m, n_1)+f(k, m, n_2)+\binom{m+1}2$
holds for all integers $k\ge 0$ and $m, n_1, n_2\ge 1$. Thus, for fixed $k$ and $m$,
the function $n\longmapsto f(k, m, n)+\binom{m+1}2$ is subadditive and, 
consequently, the limit 
\[
	\alpha(k, m)
	=
	\lim_{n\to\infty}\frac{f(k, m, n)}{n}
\]
exists. An explicit construction described in~\cite{ADR} shows 
\begin{equation}\label{eq:4100}
	\limsup_{m\to\infty}\frac{\alpha(k, m)}m
	\le 
	\sqrt{k^2+1}-k\,.
\end{equation}
As we shall explain below, our first theorem yields a complementary lower bound,
so that altogether we have 
\begin{equation}\label{eq:4141}
	\lim_{m\to\infty}\frac{\alpha(k, m)}m
	=  
	\sqrt{k^2+1}-k\,.
\end{equation}
	  
\begin{theorem}\label{thm:3}
	If $k\ge 0$ and $n\ge m\ge 1$, then 
		\[
		f(k, m, n)
		\ge
		\left((\sqrt{k^2+1}-k)m-\frac12\right)n-\frac{m^2}2\,.
	\]
	\end{theorem}

This estimate clearly implies 

\[
	\alpha(k, m)\ge (\sqrt{k^2+1}-k)m-\frac12\,,
\]
which in turn reveals
\[
	\liminf_{m\to\infty}\frac{\alpha(k, m)}m
	\ge 
	\sqrt{k^2+1}-k\,.
\]
Together with~\eqref{eq:4100} this proves indeed our claim~\eqref{eq:4141}.

\begin{proof}[Proof of Theorem~\ref{thm:3}]
	Fix a partition $[n]=A_1\dcup\dots\dcup A_{k+1}$ such that 
		\[
		f(k, m, n)
		=
		\Psi_m(A_1)+\dots+\Psi_m(A_{k+1})\,.
	\]
		For $L=n/m$ we define the partitions 
	$[0, L)=I_1\dcup\dots\dcup I_n=B_1\dcup\dots\dcup B_{k+1}$
	by setting first $I_j=\bigl[\frac{j-1}m, \frac jm\bigr)$ for every $j\in [n]$
	and then $B_i=\bigdcup_{j\in A_i} I_j$ for every $i\in [k+1]$. 
	Due to $L\ge 1$, Theorem~\ref{thm:1} 
	tells us that the sets $D_1$, $D_2$ defined in~\eqref{eq:3D1}
	and~\eqref{eq:3D2} satisfy 
		\[
		\lambda\bigl((B_1^2\dcup\dots\dcup B_{k+1}^2)\cap (D_1\dcup D_2)\bigr)
		\ge
		(\sqrt{k^2+1}-k)L\,.
	\]
		Because of $\lambda(D_2)\le \frac12$, this entails
		\[
		\sum_{i=1}^{k+1}\sum_{(x, y)\in A_i^2} \lambda(I_x\times I_y\cap D_1)		
		\ge
		\frac{(\sqrt{k^2+1}-k)n}m-\frac12\,.
	\]
		As the term $\lambda(I_x\times I_y\cap D_1)$ is
		\begin{enumerate}
		\item[$\bullet$] $0$ if $y<x$ or $y>x+m$,
		\item[$\bullet$] $1/2m^2$ if $x=y$,
		\item[$\bullet$] and at most $1/m^2$ if $x< y\le x+m$,
	\end{enumerate}
	this yields 
		\[
		\frac{\Psi_m(A_1)+\dots+\Psi_m(A_{k+1})}{m^2}
		\ge
		\frac{(\sqrt{k^2+1}-k)n}m-\frac12-\frac n{2m^2}
	\]
		and the desired inequlity follows upon multiplying by $m^2$. 
\end{proof}

\end{document}